\documentclass{article}
\usepackage{color}
\usepackage[normalem]{ulem}
\usepackage{tikz}
\usetikzlibrary{arrows}

\newcommand{\lc}{\mathrm{lc}}
\newcommand{\ord}{\mathrm{ord}\,}
\newcommand{\ini}{\mathrm{in}\,}
\newcommand{\bC}{\mathbf{C}}
\newcommand{\bK}{\mathbf{K}}

\newcommand{\bN}{\mathbf{N}}
\newcommand{\bU}{\mathbf{U}}
\newcommand{\Zer}{\mathrm{Zer}}
\newcommand{\ch}{\mathrm{Char}}
\newcommand{\cont}{\mathrm{cont}}
\newtheorem{Theorem}{Theorem}[section]

\newtheorem{Corollary}[Theorem]{Corollary}
\newtheorem{Remark}[Theorem]{Remark}
\newtheorem{Property}[Theorem]{Property}
\newtheorem{Example}[Theorem]{Example}
\newtheorem{Lemma}[Theorem]{Lemma}

\newenvironment{proof}[1][Proof]{\textbf{#1.} }{\
\rule{0.5em}{0.5em}}

\begin{document}
\title{Decompositions of the  higher order polars of plane branches
\footnotetext{
       \begin{minipage}[t]{4.5in}
       {\small
       2000 {\it Mathematics Subject Classification:\/} 
       Primary 32S05;  Secondary 32S99.\\
       Key words and phrases:  
       irreducible plane curve, higher order polar, threshold semi-root.\\
       The first-named author was partially supported by the Spanish Project
       MTM2012-36917-C03-01 and the second author was partially supported by the \em{Plan Propio de Investigaci\'on de la Universidad de La Laguna-2014}.}
       \end{minipage}}}
      
\author{Evelia R.\ Garc\'{\i}a Barroso and Janusz Gwo\'zdziewicz}

\maketitle

\begin{abstract}
 
\noindent In \cite{Casas} Casas-Alvero found decompositions of  higher order polars 
of an irreducible plane curve generalizing the results of Merle. 
We improve his result giving a finer decomposition where we determine the  topological type and the number of a kind of branches that we call {\em threshold semi-roots}.
\end{abstract}

\section{Introduction}
 \noindent In \cite{Casas} Casas-Alvero found a decomposition of  higher order polars 
of an irreducible singular plane curve. Generalizing the results of  \cite{Merle}, he
proved that the irreducible components of the higher order polar curves 
of a plane branch $f(x,y)=0$ are 
 branches that have \textit{characteristic contacts} with  $f(x,y)=0$,
 which means that their contacts with $f(x,y)=0$ are the characteristic exponents 
$b_1/b_0,\ldots,b_h/b_0$ of $f(x,y)=0$. If the contact between a branch $h(x,y)=0$ and $f(x,y)=0$ 
is equal to the characteristic exponent  $b_i/b_0$
then  
$b_1/b_0,\ldots, b_{i-1}/b_0$ are  the first characteristic exponents  of $h(x,y)=0$.

\medskip

\noindent Casas-Alvero's decomposition of the $k$th higher order polar curve of $f(x,y)=0$ involves writing
$\frac{\partial ^k}{\partial y^k}f(x,y)$ as a finite product of power series, not necessarily irreducible,
called {\it bunches}, where each bunch is in turn the product of all irreducible factors 
of  $\frac{\partial ^k}{\partial y^k}f(x,y)$ having the same contact value with $f(x,y)=0$. 

\noindent
Note that with only the information about the contact value we cannot determine 
the equisingularity type (in the sense of Zariski) of the irreducible components 
of  $\frac{\partial ^k}{\partial y^k}f(x,y)=0$ from the equisingularity type of $f(x,y)=0$. 
It is well-known that the equisingularity type of the polar curve can vary in  a family of
equisingular branches. The family $\{f_a=y^3+x^{11}+ax^8y\}_{a\in \bC}$ (see \cite[Exemple 3]{Pham}) is equisingular;  the first polar curve of $f_a(x,y)=0$ has two different smooth branches for $a\neq 0$, but it has  a double smooth branch for $a=0$.

\bigskip

\noindent In this paper we refine Casas-Alvero's decomposition. We show that every Casas-Alvero's  bunch $\Gamma$ of $\frac{\partial ^k}{\partial y^k}f(x,y)$ is the product of two power series $\Gamma_1\cdot \Gamma_2$, where 
all irreducible factors of $\Gamma_2$ called {\em threshold semi-roots}, 
have the same Puiseux characteristic depending only on the Puiseux characteristic of $f(x,y)=0$. The remaining irreducible factors of $\Gamma$ constitute $\Gamma_1$. The existence of threshold semi-roots is a new phenomenon observed for the higher order polars, because we note that the first order polar does not have such branches. We also prove that the number of Newton-Puiseux roots of $\Gamma_1=0$ and $\Gamma_2=0$ depends only on the Puiseux characteristic of $f(x,y)=0$. 

\medskip

\noindent In \cite{LMW} the authors determine the possible components of  the exceptional divisor $E$ of the minimal resolution  of the branch $f(x,y)=0$  where the strict transform of $\frac{\partial}{\partial y}f(x,y)=0$ intersects $E$. For higher order polars the result of \cite{LMW} remains true and we make it precise  for threshold semi-roots: the strict transforms of branches defined by threshold semi-roots are smooth and intersect transversely ({\it curvetta}) the rupture components of the exceptional divisor $E$ (components of $E$ intersecting at least three other components).  We observe that threshold semi-roots are not  {\it semi-roots} (in the sense of Abhyankar). 
\medskip

\noindent The decomposition theorem of the first polar of a plane reduced curve $f(x,y)=0$ allowed   to describe in \cite{GB-T} the phenomenon of {\em Lipschitz-Killing curvature concentration } on the Milnor fiber
 $f(x,y) = \lambda \subseteq \bC^2 $ for 
$ \vert (x,y)\vert< \epsilon$ when $ \lambda,\epsilon \to 0,\, \vert \lambda \vert << \epsilon$. This is a multiscale phenomenon 
(as the multiscale phenomenon shown in Example \ref{GL}) depending only on the equisingularity type of the curve. It would be expected that the decomposition of the higher order polars presented in this paper will help in the description of the metric and topological properties of the fibers of singular complex analytic morphisms.

\medskip

\noindent 
In order to refine Casas-Alvero's factorization we deal with Newton-Puiseux roots 
of~$\frac{\partial ^k}{\partial y^k}f(x,y)$. 
For any characteristic exponent $q$ of $f$ we count the number of 
roots that have a contact $q$ with $f$. Moreover our approach allows to find 
the coefficients $c_q$ of the monomial $x^q$ in these roots. 
The Newton-Puiseux roots with $c_q=0$ are  the roots of $\Gamma_1=0$ 
and the others are the roots of $\Gamma_2=0$.

\medskip
\noindent All the results of this paper remain true if we replace $\bC$ by any algebraically closed field  $\bK$ of characteristic zero.

\section{ Formal Puiseux power series}
 Denote by $\bC[[x]]^*$ the set of formal Puiseux power series. The {\em order} of any nonzero formal Puiseux power series is the minimal degree of its terms. By convention the order of the zero formal Puiseux power series is $+\infty$.
For every $\phi,\psi\in \bC[[x]]^*$ we  define $O(\phi,\psi)$ to be
the order of the difference $\phi-\psi$ and we call it 
the {\em contact order} of $\phi$ and $\psi$. 
It is well-known that for any $\phi_1,\phi_2,\phi_3\in \bC[[x]]^*$ the {\em Strong Triangle Inequality (STI)}
$O(\phi_1,\phi_3)\geq \min \{O(\phi_1,\phi_2),O(\phi_2,\phi_3)\}$ holds.
\medskip

\noindent  Let $\alpha \in\bC[[x]]^{*}$ and $r$ be a positive rational number. The set $B=\{\,\psi \in\bC[[x]]^{*}: O(\alpha,\psi) \geq r \,\}$ is called a \textit{pseudo-ball of height} $r$. Note that any two pseudo-balls of height $r$ are either disjoint or are equal. To prove it observe  that by STI if  $O(\alpha_1,\phi), O(\alpha_2,\phi), O(\alpha_1,\psi)\geq r$ then $O(\alpha_2,\psi) \geq r$. Hence if 
the pseudo-balls $\{\psi \in\bC[[x]]^{*}: O(\alpha_1,\psi) \geq r \,\}$ and $\{\psi \in\bC[[x]]^{*}: O(\alpha_2,\psi) \geq r \,\}$ have a non empty intersection  then they  are  equal.

\medskip

 \noindent Take a pseudo-ball $B$ of height $r$. Every formal Puiseux power series $\gamma(x)\in B$  has the form
$\gamma(x)=\lambda_B(x)+c_{\gamma}x^r+ \hbox{\em higher order terms}$, where $\lambda_B(x)$ is obtained from an arbitrary $\alpha(x)\in B$ by omitting all its terms of order bigger than or equal to $r$.  We call the number $c_{\gamma}$ the {\em leading coefficient of } $\gamma$ {\em with respect to} $B$ and denote it $\lc_B \gamma$. Remark that $c_{\gamma}$ can be zero.

\medskip
\noindent Hereinafter, for brevity, formal Puiseux power series will be called Puiseux series.

\section{Newton-Puiseux roots of higher order polars}

\noindent Let $f(x,y)\in \bC[[x,y]]$ be such that $1<\ord f(0,y)=n<+\infty$. Fix a positive integer $k<n$.  Then the order of  $\frac{\partial ^k}{\partial y^k} f(0,y)$ equals $n-k$.  
The Newton-Puiseux factorizations of $f(x,y)$ and $\frac{\partial ^k}{\partial y^k} f(x,y)$ have the form

\[
f(x,y)= u(x,y)\prod_{i=1}^n (y-\alpha_i(x)), 
\]
\begin{equation}
  \label{ppp2}
\frac{\partial^k}{\partial y^k}f(x,y)=\tilde u(x,y)\prod_{j=1}^{n-k} (y-\gamma_j(x)),
\end{equation}
where 
$u(x,y)$, $\tilde u(x,y)$ are units in  $\bC[[x,y]]$ and $\alpha_i(x)$, $\gamma_j(x)$
are Puiseux series of positive order called {\em Newton-Puiseux roots} of $f(x,y)=0$ and $\frac{\partial^k}{\partial y^k}f(x,y)=0$, respectively. We denote by $\Zer g$ the set of Newton-Puiseux roots of $g(x,y)=0$ for any $g(x,y)\in 
\bC[[x,y]]$.

\medskip

\noindent Let $B$ be a pseudo-ball. We put

\[
F_B(z):=\prod_{j\;:\;\alpha_j\in B}(z-\lc_B \alpha_j).
\]

\noindent Remark that the above polynomial is equal up to multiplication by a constant, 
 to the polynomial introduced in \cite[Lemma 3.3]{Gwo} (see also \cite[ Formula~(2.2)]{K-P}).

\medskip

\begin{Lemma}
\label{GL}
Let $B$ be a pseudo-ball. Assume that $k< \deg F_B(z)$. Then 
\[
\frac{d^k}{d z^k}F_B(z)= \hbox{\rm constant} \cdot \prod_{j\;:\;\gamma_j\in B}(z-\lc_B \gamma_j).
\]
\end{Lemma}

\noindent  \begin{proof}
Let $r$ be the height of $B$. Fix the weight $\omega$ such that $\omega(x)=1$, $\omega(y)=r$ 
and denote by $\ini_{\omega}(h)$ the weighted initial part of $h\in \bC[[x^{1/N},y]]$, where $N\in \bN$. 
First assume that $\lambda_B(x)=0$. Then 
\[
\ini_{\omega}f(x,y)=\hbox{constant } x^A\prod_{i\;:\alpha_i\in B}(y-\lc_B \alpha_i \cdot x^{r}),
\]
and
\[
\ini_{\omega}\frac{\partial^k}{\partial y^k}f(x,y)=\hbox{constant } x^{A'}\prod_{j\;:\gamma_j\in B}(y-\lc_B  \gamma_j \cdot x^{r}),
\]
 where $A,A'$ are rational numbers. 
 If $k\leq \deg F_B(z)$ then $\frac{\partial^k}{\partial y^k} \ini_{\omega} f(x,y)$ is nonzero and consequently $\frac{\partial^k}{\partial y^k} \ini_{\omega} f(x,y)=\ini_{\omega}\frac{\partial^k}{\partial y^k} f(x,y)$. 
For $x=1$  we get 
\[
\frac{d^k}{dy^k}\prod_{i\;:\alpha_i\in B}(y-\lc_B \alpha_i)=\hbox{constant } \cdot \prod_{j\;:\gamma_j\in B}(y-\lc_B \gamma_j ).
\]

\noindent If $\lambda_B(x)\neq 0$ then taking $g(x,y):=f(x,y+\lambda_B(x))$
we reduce the proof to the first case. 
\end{proof}

\section{Properties of branches}

\noindent Denote by $\bU_m$ the multiplicative group of the $m$th complex roots of unity. This group acts on $\bC[[x^{1/m}]]$  in the following way: for  $\epsilon \in \bU_m$  and $\alpha=\sum_ia_i x^{i/m}$  

\begin{equation}
\label{star}
\epsilon *_m\alpha=\sum_ia_i\epsilon^ix^{i/m}.
\end{equation}

\noindent The {\em star operation} defined in (\ref{star})  preserves the contact, that is
$O(\alpha_1,\alpha_2)=O(\epsilon*_m\alpha_1,\epsilon*_m\alpha_2)$.
\medskip

\noindent Let $\alpha$ be a Puiseux series. The smallest natural number $n$ such that $\alpha\in \bC[[x^{1/n}]]$ is called the {\em index} of $\alpha$. Denote by $*$  the star operation of $\bU_n$ on $\bC[[x^{1/n}]]$ introduced in (\ref{star}).
\noindent 
Observe that if the  Puiseux series $\alpha$ has index $n$ then  $\epsilon_1 *\alpha\neq \epsilon_2 *\alpha$, for any two different $n$-th roots of the unity  $\epsilon_1,\epsilon_2$ (see \cite[Lemma 3.9]{Hefez}).

\medskip

\noindent For  a Puiseux series $\alpha=\sum_ia_i x^{i/n}$ of positive order  and index $n$ we introduce two sequences $(e_i)$ and $(b_i)$ of natural numbers as follows:
\begin{itemize}
\item $e_0=b_0=n$,
\item if $e_k\neq 1$ then 
$b_{k+1}:=\min \{i\;:\;  i\not\equiv 0\; \hbox{\rm mod } e_{k}\; \mbox{and } a_i\neq 0\}$,
\item $e_k=\gcd(e_{k-1},b_k)$.
\end{itemize}

\noindent The sequence $e_i$ is strictly decreasing and for some $h\in \bN$ we have $e_h=1$. We get
$\bU_n=\bU_{e_0} \supset \bU_{e_1}  \supset \cdots  \supset \bU_{e_h}=\{1\}$. After \cite[Lemma 6.8]{Hefez} if $\epsilon \in \bU_{e_{k-1}}\backslash  \bU_{e_{k}}$ then $\epsilon^{b_k}\neq 1$. Consequently 
\begin{equation}
\label{contact}
O(\alpha,\epsilon*\alpha)=\frac{b_k}{n} \;\;\hbox{\rm for } \epsilon \in \bU_{e_{k-1}}\backslash \bU_{e_k}.
\end{equation}

\medskip

\noindent  Let  $\alpha$ be a  Puiseux series  of index $n$ which is a Newton-Puiseux root of an irreducible power series $f(x,y)\in \bC[[x,y]]$. Then $\Zer f=\{\epsilon * \alpha \;:\; \epsilon \in \bU_n\}$ and consequently $\ord f(0,y)=n$ (see \cite[Theorem 3.10]{Hefez}). The {\em characteristic} of an irreducible power series $f(x,y)\in \bC[[x,y]]$ is the sequence $(b_0,b_1,\ldots,b_h)$,  associated to any Newton-Puiseux root of $f$. By (\ref{contact}) the set $\ch f:=\left \{\frac{b_1}{b_0},\cdots,  \frac{b_h}{b_0}\right\}$ is the set of contacts between the Newton-Puiseux roots of $f$. We call  $\ch f$ the set of {\em characteristic exponents} of $f$.

\medskip

\noindent Let $T_i(f)$ be the set of pseudo-balls of height $\frac{b_i}{b_0}$ having non-empty intersection with $\Zer f$.

\begin{Property}\label{PR2}
For every characteristic exponent $b_i/b_0$ the set $T_i(f)$ consists of $e_0/e_{i-1}$ 
pairwise disjoint pseudo-balls. 
Every $B\in T_i(f)$  contains $e_{i-1}$ elements of $\Zer f$ and 
$F_B(z)=(z^{e_{i-1}/e_i}-c_B)^{e_i}$ for some $c_B\neq 0$.
\end{Property}

\noindent \begin{proof} Let  $B\in T_i(f)$  and 
 $\alpha \in B\cap \Zer f$. By (\ref{contact}) $B\cap \Zer f=\{\epsilon * \alpha \;:\; \epsilon \in \bU_{e_{i-1}}\}$, which shows that $B$ contains $e_{i-1}$ elements of $\Zer f$. Consequently $T_i(f)$ consists of $e_0/e_{i-1}$ 
pairwise disjoint pseudo-balls.

\noindent We get 
\[
F_B(z)=
\prod_{\epsilon^{e_{i-1}}=1}(z-\lc_B(\epsilon * \alpha))=
\prod_{\epsilon^{e_{i-1}}=1}(z-\epsilon^{b_i}a_{b_i})=
(z^{e_{i-1}/e_i}-a_{b_i}^{e_{i-1}/e_i})^{e_i},
\]
where $a_{b_i}$ is the coefficient of the monomial $x^{b_i/b_0}$ 
of $\alpha$. The last equality follows from \cite[Lemma 3.4]{Acta}.
\end{proof}

\medskip

\newpage

\begin{Example}
\label{GL}
Consider the irreducible  complex convergent power series
$f(x,y)=((y^3-x^4)^4+x^{17}y^3)^2+x^{22}(y^3-x^4)^5$ 
of characteristic $(24,32,62,137)$.  
Let $\alpha_i(x)$, $i=1,\dots, 24$ be the Newton-Puiseux roots of $f(x,y)=0$. 

\noindent Take $\epsilon>0$ small enough.
A higher contact order between $\alpha_i(x)$, $\alpha_j(x)$ means a smaller Euclidean distance 
between $\alpha_i(\epsilon)$, $\alpha_j(\epsilon)$. Thus the pseudo-balls of~$T_i(f)$, $i=1,2,3$ correspond to groups of roots of $f(\epsilon,y)=0$. 
These roots, for $\epsilon=0.75$, are drawn on the left side of  Figure 1.  

\noindent Fix $B\in T_2(f)$ and $E\in T_3(f)$. 
Using \cite[Lemma 3.3]{Gwo} one can show that there are constants 
$C_{\epsilon},D_{\epsilon}\in \bC $ such that,  
\begin{equation}
\label{convergence}
\begin{array}{ll}
C_{\epsilon}\cdot f(\epsilon, \lambda_B(\epsilon)+z\cdot \epsilon^{62/24}) \rightarrow F_B(z),  \;&
\\
D_{\epsilon}\cdot f(\epsilon, \lambda_E(\epsilon)+z\cdot \epsilon^{137/24}) \rightarrow F_E(z),  \;&\end{array}
\end{equation}
when  $\epsilon \rightarrow 0$.
This asymptotic property is illustrated on the right side of  Figure 1. 
Notice that by Property~\ref{PR2} we have 
$F_B(z)=(z^4-c_1)^2$ and $F_E(z)=z^2-c_2$ for some nonzero constants $c_1$, $c_2$. 

\begin{center}
\begin{tikzpicture}[scale=0.35]
\draw[->,thin] (-11,0)--(11,0) ;
\draw[->,thin] (0,-11)--(0,11) ;
\node[draw,circle,inner sep=0.6pt,fill=black] at (-4.984,-6.375){};  
\node[draw,circle,inner sep=0.6pt,fill=black] at (-4.984,6.375){};  
\node[draw,circle,inner sep=0.6pt,fill=black] at (-4.739,-6.348){};  
\node[draw,circle,inner sep=0.6pt,fill=black] at (-4.734,6.348){};  
\node[draw,circle,inner sep=0.6pt,fill=black] at (-3.983,-4.477){};  
\node[draw,circle,inner sep=0.6pt,fill=black] at (-3.983,4.477){}; 
\node[draw,circle,inner sep=0.6pt,fill=black] at (-3.719,-4.132){};  
\node[draw,circle,inner sep=0.6pt,fill=black] at (-3.719,4.132){}; 
\node[draw,circle,inner sep=0.6pt,fill=black] at (-3.127,-7.279){};  
\node[draw,circle,inner sep=0.6pt,fill=black] at (-3.127,7.279){};  
\node[draw,circle,inner sep=0.6pt,fill=black] at (-3.029,-7.504){};  
\node[draw,circle,inner sep=0.6pt,fill=black] at (-3.029,7.504){};  
\node[draw,circle,inner sep=0.6pt,fill=black] at (-1.885,-5.687){};  
\node[draw,circle,inner sep=0.6pt,fill=black] at (-1.885,5.687){}; 
\node[draw,circle,inner sep=0.6pt,fill=black] at (-1.719,-5.287){};  
\node[draw,circle,inner sep=0.6pt,fill=black] at (-1.719,5.287){}; 
\node[draw,circle,inner sep=0.6pt,fill=black] at (5.438,-1.154){};  
\node[draw,circle,inner sep=0.6pt,fill=black] at (5.438,1.154){};  
\node[draw,circle,inner sep=0.6pt,fill=black] at (5.868,-1.212){};  
\node[draw,circle,inner sep=0.6pt,fill=black] at (5.868,1.212){};  
\node[draw,circle,inner sep=0.6pt,fill=black] at (7.867,-0.931){};  
\node[draw,circle,inner sep=0.6pt,fill=black] at (7.867,0.931){}; 
\node[draw,circle,inner sep=0.6pt,fill=black] at (8.013,-1.129){};  
\node[draw,circle,inner sep=0.6pt,fill=black] at (8.013,1.129){}; 
\draw [olive](0,0) circle (9.8cm);
\draw [blue](6.75,0) circle (2.5cm);
\draw [blue](-3.35,5.82) circle (2.5cm);
\draw [blue](-3.35,-5.82) circle (2.5cm);
\draw [red](5.64,1.18) circle (0.5cm);
\draw [red](5.64,-1.18) circle (0.5cm);
\draw [red](7.95,1.05) circle (0.5cm);
\draw [red](7.95,-1.05) circle (0.5cm);
\draw [red](-4.85,6.35) circle (0.5cm);
\draw [red](-3.84,4.30) circle (0.5cm);
\draw [red](-3.08,7.39) circle (0.5cm);
\draw [red](-1.8,5.48) circle (0.5cm);
\draw [red](-4.85,-6.35) circle (0.5cm);
\draw [red](-3.84,-4.30) circle (0.5cm);
\draw [red](-3.08,-7.39) circle (0.5cm);
\draw [red](-1.8,-5.48) circle (0.5cm);
\draw [olive](5.1,9.4) node[above]{$T_1(f)=\hbox{\rm the olive ball}$} ;
\draw [blue](5,4.8) node[above]{$T_2(f)=\hbox{\rm 3 blue balls}$};
\draw [red](5,2.8) node[above]{$T_3(f)=\hbox{\rm 12 red balls}$} ;
\draw [blue](5.6,-3.5) node[above]{$B$} ;
\draw [red](7.3,1.1) node[above]{$E$} ;
\draw (15,-3.5) node[above]{$\hbox{\rm roots of }F_B$} ;
\draw (20.5,-3.5) node[above]{$\hbox{\rm roots of }F_E$} ;
\draw (6.5,-13.5) node[above]{\hbox{\rm Figure 1}} ;
\draw[->,thin] (13,0)--(17,0) ;
\draw[->,thin] (15,-2)--(15,2) ;
\node[draw,red, circle,inner sep=1pt,fill=red] at (15.7,0.7){}; 
\node[draw,red, circle,inner sep=1pt,fill=red] at (15.7,-0.7){}; 
\node[draw,red, circle,inner sep=1pt,fill=red] at (14.3,0.7){}; 
\node[draw,red, circle,inner sep=1pt,fill=red] at (14.3,-0.7){}; 
\draw [blue](15,0) circle (1.5cm);
\draw[->,thin] (18,0)--(22,0) ;
\draw[->,thin] (20,-2)--(20,2) ;
\node[draw,circle,inner sep=1pt,fill] at (19.75,-0.35){}; 
\node[draw,circle,inner sep=1pt,fill] at (20.25,0.35){}; 
\draw [red](20,0) circle (1cm);
\end{tikzpicture}
\end{center}

\noindent The convergence in (\ref{convergence}) is almost uniform. Hence there are similar limits for higher derivatives. This explains  why we could detect the position of the roots of the $k$th derivative of $f$ by the position of the roots of the polynomials $F^{(k)}_B(z)$.

\end{Example}

\subsection{Contact of branches}

\noindent Let $f, g\in\bC[[x,y]]$ be  irreducible power series
coprime with $x$.
For every Puiseux series $\gamma$ we define 
$\cont(f,\gamma)=\max\{ O(\alpha,\gamma):\alpha\in\Zer f\}$
and call this number the {\em contact} between $f$ and $\gamma$. By abuse of notation we put 
\[ \cont(f,g)=\max\{ O(\alpha,\gamma):\alpha\in\Zer f,\gamma\in\Zer g\}. \]

\noindent In this section we take $m\in \bN$ such that $\Zer f,\Zer g\subset \bC[[x^{1/m}]]$  and we consider the star operation $*_m$ of $\bU_m$ in $\bC[[x^{1/m}]]$ introduced in (\ref{star}). If $\alpha=\sum_ia_i x^{i/n}$ is a Newton-Puiseux root of $f(x,y)=0$ of index $n$ then $m=q n$ for some $q \in \bN$. Then $\alpha=\sum_ia_i x^{i q/m}$ and $\theta *_m\alpha=\theta^q *_n \alpha$, where $*_n$ is the star operation of $\bU_n$ on $\bC[[x^{1/n}]]$. Since  $\bU_n=\{\theta^q\;:\;\theta\in \bU_m\}$, the action of $\bU_m$ permutes $\Zer f$ and  for every $\alpha, \alpha'\in \Zer f$ there exists $\epsilon \in \bU_m$ such that $\alpha'=\epsilon *_m \alpha$. Up to the end of this section we denote $*_m$ by $*$.

\begin{Property}
\label{contact-roots}
For every $\gamma\in\Zer g$\quad $\cont(f,\gamma)=\cont(f,g)$. 
\end{Property}

\noindent \begin{proof}
It is enough to show that for all $\gamma,\gamma'\in\Zer g$ the sets 
of contact orders
$\{O(\alpha,\gamma):\alpha\in\Zer f\}$ and 
$\{O(\alpha,\gamma'):\alpha\in\Zer f\}$ are equal. 

\noindent Take $\epsilon\in \bU_m$ such that $\gamma'=\epsilon* \gamma$. 
Then $O(\alpha,\gamma)=O(\epsilon*\alpha,\gamma')$ for all $\alpha\in\Zer f$.
Since the action of $\bU_m$ permutes $\Zer f$ the sets  under consideration are equal.
\end{proof}

\begin{Property}
\label{contact-exponents}
For every $q<\cont(f,g)$\quad $ q\in\ch f$ if and only if $q\in\ch g$.
\end{Property}

\noindent \begin{proof}
Let $q<\cont(f,g)$ be a characteristic exponent of $f$. By definition of 
the characteristic exponent $O(\alpha,\alpha')=q$ for some $\alpha, \alpha'\in\Zer f$. 
Following Property~\ref{contact-roots}
$\cont(g,\alpha)=\cont(g,\alpha')=\cont(g,f)$. 
Hence there exist $\gamma,\gamma'\in\Zer g$ such that 
$O(\gamma,\alpha)=O(\gamma',\alpha')=\cont(g,f)=\cont(f,g)$. 
By STI we get 
$O(\gamma,\gamma')\geq\min\{O(\gamma,\alpha),O(\alpha,\alpha'),O(\alpha',\gamma')\}=q$. 
Suppose that $O(\gamma,\gamma')>q$. Then  we would have 
$O(\alpha,\alpha')\geq\min\{O(\alpha,\gamma),O(\gamma,\gamma'), O(\gamma',\alpha')\}>q$
which is absurd. Hence $q=O(\gamma,\gamma')$ is a characteristic exponent of $g$. 
\end{proof}

\medskip

 \noindent Let $\alpha=\sum_ia_ix^{i/n} \in \bC[[x]]^*$ be a Puiseux  series. The {\em support of} $\alpha$ is the set $\{i/n\;:\;a_i\neq 0\}$.

\begin{Property}
\label{contact-nonzero}
If $q=\cont(f,g)$ is a characteristic exponent of $f$ and there exists 
a  Puiseux series $\gamma\in\Zer g$ such that $q$ is in the support of $\gamma$  then $q$ is a characteristic exponent of $g$.
\end{Property}

\noindent \begin{proof}
Take $\alpha,\alpha'\in\Zer f$ such that $O(\alpha,\gamma)=O(\alpha,\alpha')=q$ and  let
 $\epsilon\in \bU_m$ be such that $\alpha'=\epsilon * \alpha$. 
Put $\gamma'=\epsilon *\gamma$. 

\noindent By STI we get 
$O(\gamma,\gamma')\geq\min\{O(\gamma,\alpha),O(\alpha,\alpha'),O(\alpha',\gamma')\}=q$. 
The equality $O(\alpha,\alpha')=q$ implies that $\epsilon *x^q\neq x^q$. 
Thus the monomial $x^q$ appears in the difference $\gamma-\epsilon * \gamma$ with a nonzero 
coefficient which proves that $O(\gamma,\gamma')=q$. Therefore $q\in\ch g$.
\end{proof}

\section{Roots of derivatives of special polynomials}

 In this section we study the roots of the complex polynomial  
$\frac{d^k}{dz^k}(z^n-c)^e$.

\begin{Property}\label{P1}
Let $F(t)=H(t^n)$ be a complex polynomial. 
If $t_0$ is a nonzero root of~$F(t)$ of multiplicity $m$ then $(t^n-t_0^n)^m$ divides $F(t)$.
\end{Property}
\begin{proof}
It is enough to factorize $F(t)$ in the ring $\bC[t^n]$ and notice that $t_0$ is a 
root of a factor  $t^n-a$ if and only if $a=t_0^n$.
\end{proof}

\begin{Lemma}\label{L1}
Let $F(t)$ be a real polynomial of positive degree of the form 
\[F(t)=C \cdot t^{a}(t^n-1)^{b}\prod_{i=1}^{d}(t^n-c_i), \]

\noindent where $a$, $b$, $d$ are nonnegative integers 
and $c_i$ are pairwise distinct real numbers from the interval $(0,1)$. 
Then the derivative of $F(t)$ has the form

\[F'(t)=C' \cdot t^{a'}(t^n-1)^{b'}\prod_{i=1}^{d'}(t^n-c_i'),\]

\noindent where $c_i'$ are pairwise distinct real numbers from the interval $(0,1)$.
Moreover:
if $a>0$ then $a'=a-1$, 
if $a=0$ then $a'=n-1$,
if $b>0$ then $b'=b-1$,
if $b=0$ then $b'=0$.
\end{Lemma}
\begin{proof}
Let $a'$ be the multiplicity of $0$ as a root of $F'(t)$, 
let $b'$ be the multiplicity of $1$ as a root of $F'(t)$
and let $d'$ be the number of distinct real roots of $F'(t)$ in the interval~$(0,1)$. 
We will check that 
\begin{equation}\label{E2}
(a'-a)+[(b'-b)+(d'-d)]n \geq -1 
\end{equation}

\noindent which is equivalent with
\begin{equation}\label{E1}
a'+(b'+d')n \geq \deg F'(t) .
\end{equation}

\medskip
\noindent Consider several cases 
depending on the values of $a$ and $b$.\\

\noindent (I) $a,b>0$.
The polynomial $F(t)$ has $d+2$ distinct real roots in the closed interval $[0,1]$. 
These roots divide $[0,1]$ to $d+1$ sub-intervals. 
By Rolle's Theorem inside each sub-interval there is at least one root of $F'(t)$. 
Hence $d'\geq d+1$. 
The differentiation decreases the multiplicity of a root of a polynomial by $1$. 
Thus  $a'=a-1$ and $b'=b-1$.  \\
(II) $a>0$ and $b=0$. 
By similar arguments as before we get $a'=a-1$, $b'\geq 0$, and $d'\geq d$.\\
(III) $a=0$ and $b>0$. 
In this case $F(t)$ is a polynomial of $t^n$. Taking the derivative we get $a'\geq n-1$. 
Moreover $b'=b-1$ and $d'\geq d$.\\
(IV) $a=b=0$. 
We get $a'\geq n-1$, $b'\geq 0$ and $d'\geq d-1$.\\

\noindent One easily verifies that inequality~(\ref{E2}) holds in each case. 

\medskip

\noindent Let $t_1$, \dots, $t_{d'}$ be the pairwise distinct real roots of $F'(t)$ from the interval $(0,1)$. 
Consider the polynomial 
$P(t)=t^{a'}(t^n-1)^{b'}\prod_{i=1}^{d'}(t^n-t_i^n)$. 
Inequality~(\ref{E1}) reads $\deg P(t)\geq \deg F'(t)$. 
By Property~\ref{P1}  $P(t)$ divides $F'(t)$.  
Hence $P(t)$ and $F'(t)$ are equal up to a multiplication by a constant. 
This proves the first statement of the lemma. 

\medskip

\noindent By the form of $F'(t)$ we get $\deg F'(t)=a'+(b'+d')n$. 
Consequently  the inequality ``$\geq$'' in ~(\ref{E2}) can be replaced by ``$=$''.
This implies that all weak inequalities obtained in the above case-by-case analysis must be equalities and proves the second statement of the lemma. 
\end{proof}

\begin{Lemma}
\label{AL}
Let $F(z)=(z^n-c)^{e}$ be a complex polynomial.  
Then for $1\leq k<\deg F(z)$ one has 
$\frac{d^k}{dz^k}F(z)=C z^{a}(z^n-c)^{b}\prod_{i=1}^{d}(z^n-c_i)$, where $C\in \bC$ and
\begin{enumerate}
\item[(1)] $0\leq a <n$ and $a+k\equiv 0 \pmod n$,
\item[(2)] $b=\max\{e-k,0\}$,
\item[(3)] $d=\min\{e,k\}-\lceil \frac{k}{n} \rceil$,  where $\lceil x \rceil$ denotes the smallest integer number larger than or equal to $x$,
\item[(4)] $a +(b+d)n=ne-k$,
\item [(5)] $c_i\neq c_j$ for $1\leq i<j\leq d$ and $0\neq c_i\neq c$ for $1\leq i\leq d$.
\end{enumerate}
\end{Lemma}
 \begin{proof}
Without loss of generality we may assume that $c=1$. The general case
reduces to this case by replacing $F(z)$ by  the polynomial 
$c^{-e}F(\sqrt[n]{c}z)$. 

\medskip

 \noindent The polynomial $F(z)=(z^n-1)^e$ satisfies the assumptions of  Lemma \ref{L1}. 
Applying this lemma  to subsequent derivatives of $F(z)$ we see
that the  $k$th derivative of $F(z)$ has the form 
$\frac{d^k}{dz^k}F(z)=C_k z^{a_k}(z^n-1)^{b_k}\prod_{i=1}^{d_k}(z^n-c_{i,k})$
and verifies the assumptions of Lemma~\ref{L1}, for $0\leq k<\deg F(z)$. 
This implies (4) and~(5). By Lemma~\ref{L1} we get $0\leq a_k \leq n-1$. 
The congruence $a_k+k\equiv 0 \pmod n$ is a consequence of (4), and we conclude (1). 
Remark that $b_0=e$ and by Lemma~\ref{L1} we have $b_{i+1}=b_i-1$ if $b_i>0$ 
and  $b_{i+1}=0$ if $b_i=0$. This gives~(2). Now we will prove (3).  By (2) we get $e-b_k=\min \{e,k\}$ and by (1) the quotient $\frac{a_k+k}{n}$ is the smallest integer number larger than or equal to $\frac{k}{n}$. Computing $d_k$ from (4) we get $d_k=e-b_k-\frac{a_k+k}{n}=\min\{e,k\}-\lceil \frac{k}{n} \rceil$.
\end{proof}

\begin{Corollary}
\label{CC}
Let $F(z)=(z^n-c)^{e}$ be a complex polynomial. 
Then every nonzero derivative $\frac{d^k}{dz^k}F(z)$ has no multiple complex roots 
except $0$ and the roots of $z^n=c$. 
\end{Corollary}

\section{Higher order polars of a branch}
 
 \noindent Let $i_k$ be the nonnegative  number such that $e_{i_k}\leq k < e_{i_k-1}$. Remember that  the $k$th
 partial derivative of $f(x,y)$ admits a decomposition $\frac{\partial^k} {\partial y^k}f(x,y)=\hbox{\rm unit }\prod_{j=1}^{n-k}(y-\gamma_j)$, where $\gamma_j$ are Puiseux series of positive order.

\medskip

\noindent Hereafter the notation $\left[a \;\;\hbox{\rm mod } n\right]$ for an integer $a$ and a natural number $n$ means the remainder of the division of $a$ by $n$.

\begin{Lemma}
\label{TL}
Let $f(x,y)\in \bC[[x,y]]$ be an irreducible power series  of characteristic $(b_0,b_1,\ldots,b_h)$.
\begin{enumerate}
\item For every $\gamma_j\in \Zer \frac{\partial^k} {\partial y^k}f$ we have $\cont(f,\gamma_j)\in \left \{\frac{b_1}{b_0},\ldots, \frac{b_{i_k}}{b_0}\right\}$.
\item If $i< i_k$ then the number of $\gamma_j$'s with $\cont(f,\gamma_j)=\frac{b_i}{b_0}$ equals $\left(\frac{b_0}{e_i}-\frac{b_0}{e_{i-1}}\right)k$.
\item  If $i= i_k$ then  the number of $\gamma_j$'s with $\cont(f,\gamma_j)=\frac{b_i}{b_0}$ equals  ${b_0}-\frac{b_0}{e_{i-1}}k$.
\item If $i\leq i_k$ then the number of $\gamma_j$'s with $\cont(f,\gamma_j)=\frac{b_i}{b_0}$ and such that ${b_i/b_0}$ is not in the support of $\gamma_j$ equals $\frac{b_0}{e_{i-1}}\left[-k \;\;\hbox{\rm mod } n_i\right]$, where $n_i=\frac{e_{i-1}}{e_i}$.
\item If $i\leq i_k$ then the number of $\gamma_j$'s with $\cont(f,\gamma_j)=\frac{b_i}{b_0}$ and such that ${b_i/b_0}$ is in the support of  $\gamma_j$  equals $\frac{b_0}{e_{i}}\left(\min\{e_i,k\}-\lceil \frac{k}{n_i}\rceil\right)$.
\end{enumerate}

\end{Lemma}
\begin{proof} Let $i\leq i_k$. Take  a pseudo-ball $B\in T_i(f) $. After Property \ref{PR2} the polynomial $F_B(z)$ has the form $(z^{n_i}-c)^{e_i}$ for some $c\neq 0$, so its degree is 
$e_{i-1}$ and after the choice of $i$ we get $k< e_{i_k-1}\leq e_{i-1}=\deg F_B(z)$. By Lemma \ref{AL} we get
\begin{equation}
\label{ad}
\frac{d^k}{dz^k}F_B(z)=C z^{a}(z^{n_i}-c)^{b}\prod_{j=1}^{d}(z^{n_i}-c_j). 
\end{equation}

\noindent By Lemma \ref{GL} the number of $\gamma_j$'s in $B$ such that $\cont(f,\gamma_j)=\frac{b_i}{b_0}$ equals 
$i_B:=a+dn_i$.  After Lemma \ref{AL}  we have 
\[
 i_B=\left\{
\begin{array}{rl}
(n_i-1)k & \hbox{\rm if } i< i_k\\
e_{i_k-1}-k & \hbox{\rm if } i=i_k.\\
\end{array}
\right.
\]

\noindent For every $\gamma_j$ satisfying $\cont(f,\gamma_j)=\frac{b_i}{b_0}$ there is a unique pseudo-ball $B\in T_i(f)$ containing $\gamma_j$. By Property \ref{PR2} the total number of $\gamma_j$'s with $\cont(f,\gamma_j)=\frac{b_i}{b_0}$ equals $i_B\frac{b_0}{e_{i-1}}$. This gives the second and third statements. As a consequence the number of $\gamma_j$'s with $\cont(f,\gamma_j)\in \left \{\frac{b_1}{b_0},\ldots, \frac{b_{i_k}}{b_0}\right\}$
equals
\[
\sum_{i=1}^{i_k-1} \left(\frac{b_0}{e_i}-\frac{b_0}{e_{i-1}}\right)k+{b_0}-\frac{b_0}{e_{i_k-1}}k=b_0-k,
\]

\noindent which is the total number of Newton-Puiseux roots of  $\frac{\partial^k} {\partial y^k}f(x,y)=0$. This proves the first statement.

\medskip

\noindent Given $B\in T_i(f)$, consider all $\gamma_j\in B$ such that $\cont(f,\gamma_j)=b_i/b_0$. By Lemma$\,$\ref{GL} the number of such $\gamma_j$'s with $\lc_B\gamma_j=0$ equals $a$ while the number of such $\gamma_j$'s with $\lc_B\gamma_j\neq 0$ equals $n_id$, where $a$ and $d$ are from (\ref{ad}).

\medskip

\noindent Recall that there are $\frac{b_0}{e_{i-1}}$ pseudo-balls in $T_i(f)$. We finish the proof of the last two statements computing the values $\frac{b_0}{e_{i-1}}a$ and $\frac{b_0}{e_{i-1}}n_id=\frac{b_0}{e_{i}}d$ using the first and the third items of Lemma \ref{AL}.
\end{proof}

\medskip

 \noindent The next theorem is an improvement of  \cite[Theorem 3.1]{Casas}.

\begin{Theorem}
\label{main}
Let $f(x,y)\in \bC[[x,y]]$ be an irreducible power series  of cha\-rac\-teristic $(b_0,b_1,\ldots,b_h)$. Put $e_s=\gcd(b_0,\ldots,b_s)$.
Fix $1\leq k<\ord f(0,y)$, and let $i_k$ be the nonnegative integer number such that $e_{i_k}\leq k < e_{i_{k-1}}$. Then $\frac{\partial^k} {\partial y^k}f(x,y)$ admits a factorization as following:
\[\frac{\partial^k} {\partial y^k}f(x,y)=\Gamma^{(1)}\cdots \Gamma^{(i_k)},\]

\noindent where $\Gamma^{(i)}$ are power series, not necessarily irreducible, verifying: 
\begin{enumerate}
\item  For each $1\leq i\leq i_k$, all branches of $\Gamma^{(i)}$ have contact $b_i/b_0$ with $f(x,y)=0$. The order of   $\Gamma^{(i)}(0,y)$  equals $\left(\frac{b_0}{e_i}-\frac{b_0}{e_{i-1}}\right)k$, for $i< i_k$ and ${b_0}-\frac{b_0}{e_{i-1}}k$ for $i=i_k$.
\item $\Gamma^{(i)}$ can be written as a product $\Gamma^{(i)}_1\Gamma^{(i)}_2$, where for any irreducible factor $g$ of $\Gamma^{(i)}_1$ the first $i-1$ characteristic exponents of $f$ and $g$ are the same  
and 
$\frac{b_i}{b_0}\not\in \ch \, g$; and $\left \{\frac{b_1}{b_0},\ldots,\frac{b_{i}}{b_0}\right\}$ is the set of characteristic exponents of any irreducible factor of $\Gamma_2^{(i)}$.
\item The order of $\Gamma_1^{(i)}(0,y)$  equals $\frac{b_0}{e_{i-1}}\left[-k \;\;\hbox{\rm mod } n_i\right]$, where $n_i=\frac{e_{i-1}}{e_i}$.
\item The order of $\Gamma_2^{(i)}(0,y)$  equals $ \frac{b_0}{e_{i}}\left(\min\{e_i,k\}-\lceil \frac{k}{n_i}\rceil\right)$.
\item The power series $\Gamma_2^{(i)}$ has $\min\{e_i,k\}-\lceil \frac{k}{n_i}\rceil$ irreducible factors.  
\end{enumerate}
\end{Theorem}

\noindent \begin{proof}

\noindent  We factorize  $\prod_{j=1}^{n-k} (y-\gamma_j)$  into $\bar{\Gamma}^{(1)}\cdots\bar{\Gamma}^{(i_k)}$, where  every $\bar\Gamma^{(i)}$ is the product $\prod (y-\gamma_j)$ running over $\gamma_j\in \Zer \frac{\partial^k} {\partial y^k}f$ with $\cont(f,\gamma_j)=\frac{b_i}{b_0}$. By  Property \ref{contact-roots} if $g$ is an irreducible  factor of $\frac{\partial^k} {\partial y^k}f$ and $\gamma,\gamma' \in \Zer g$ then $\cont(f,\gamma)=\cont(f,\gamma')$. Hence 
$\bar\Gamma^{(i)}$ is the product of irreducible power series, so it is a power series. By$\,$(\ref{ppp2}) we have
\[\frac{\partial^k} {\partial y^k}f(x,y)=\Gamma^{(1)}\cdots \Gamma^{(i_k)},\]
\noindent where $\Gamma^{(i)}$ equals $\bar\Gamma^{(i)}$,  up to multiplication by a unit, for $1\leq i\leq i_k$.

 \medskip
 
 \noindent The first statement of the theorem is a consequence of the first three statements of Lemma \ref{TL}, since the order of $\Gamma^{(i)}(0,y)$  is the number of  $\gamma_j$'s with $\cont(f,\gamma_j)=\frac{b_i}{b_0}$.
 
 \medskip

 \noindent By Property \ref{contact-exponents}  
  for any irreducible factor $g$ of $\Gamma^{(i)}$ the first $i-1$ characteristic exponents of $f$ and $g$ are the same. We define $\Gamma^{(i)}_1$  (respectively $\Gamma^{(i)}_2$) as the product of all irreducible factors $g$  of $\Gamma^{(i)}$ such that $\frac{b_i}{b_0}$ is not in the support of the Newton-Puiseux roots of $g$ (respectively $\frac{b_i}{b_0}$ is in the support of  the Newton-Puiseux roots of $g$). We will prove that  $\left \{\frac{b_1}{b_0},\ldots,\frac{b_{i}}{b_0}\right\}$ is the set of characteristic exponents of any irreducible factor  of $\Gamma_2^{(i)}$. By Property \ref{contact-nonzero}, the first $i$ charac\-te\-ris\-tic exponents of $g$ are $\frac{b_1}{b_0},\ldots, \frac{b_i}{b_0}$.
  Suppose that there exists a rational number $r> \frac{b_i}{b_0}$ which is also a characteristic exponent of $g$. Let $\gamma$ be a Newton-Puiseux root of $g$. By  (\ref{contact}) there is $\gamma'\in Zer g$, $\gamma'\neq \gamma$ such that $O(\gamma,\gamma')=r > \frac{b_i}{b_0}$. Consider the pseudo-ball $B\in T_i(f)$ containing $\gamma$. We get $\lc_B \gamma=\lc_B\gamma'$. Hence $\lc_B \gamma$ is a multiple root of $\frac{d^k}{dz^k}F_B(z)$, which  contradicts Corollary \ref{CC}.
 
  \medskip
  
 \noindent The third and fourth statements of the theorem are a consequence of the fourth and fifth  items of Lemma \ref{TL}.
 
  \medskip
  
 \noindent For any irreducible factor $g$ of $\Gamma_2^{(i)}$ the order of
 $g(0,y)$   is the least common denominator of the elements of 
 $\ch \, g=\left \{\frac{b_1}{b_0},\ldots,\frac{b_{i}}{b_0}\right\}$, that is $\frac{b_0}{e_i}$.
 The number of irreducible factors of the  power series $\Gamma_2^{(i)}$ is the quotient of the
 order of   $\Gamma_2^{(i)}(0,y)$ by $\frac{b_0}{e_i}$, which finishes the proof.  
\end{proof}

\bigskip

\noindent The first part of Theorem \ref{main} is \cite[Theorem 3.1]{Casas}. For $k=1$ the power series $\Gamma_2^{(i)}$ is a unit and consequently $\Gamma^{(i)}=\Gamma_1^{(i)}$ for every factor of Casas-Alvero's decomposition.

\medskip

\begin{Corollary}
\label{coro1}
With the notations and assumptions of Theorem~\ref{main}:
\begin{enumerate}
\item If $k=e_{i-1}-1$ then $\Gamma^{(i)}$ is irreducible and 
$\ch \, \Gamma^{(i)}=\left \{\frac{b_1}{b_0},\ldots,\frac{b_{i-1}}{b_0}\right\}$.

\item If $k=e_{i-1}-n_i$ then $\Gamma^{(i)}$ is irreducible and 
$\ch \, \Gamma^{(i)}=\left \{\frac{b_1}{b_0},\ldots,\frac{b_{i}}{b_0}\right\}$.
\end{enumerate}
\end{Corollary}

\noindent\begin{proof}
If $k=e_{i-1}-1$ then by the first and the third statements of Theorem \ref{main} we get $\ord  \Gamma_1^{(i)}(0,y)=\ord \Gamma^{(i)}(0,y)=n_1\cdots n_{i-1}$.

\noindent On the other hand, by the second statement of Theorem \ref{main}, the  first $i-1$ characteristic exponents of any branch $g$ of 
$\Gamma_1^{(i)}$ are $\frac{b_1}{b_0},\ldots,\frac{b_{i-1}}{b_0}$, and consequently the order of $g(0,y)$ is greater than or equal to $n_1\cdots n_{i-1}$. So there exists a unit $u\in \bC[[x,y]]$ such that $\Gamma^{(i)}=ug$ and $\Gamma^{(i)}$ is irreducible with 
$\ch \, \Gamma^{(i)}=\left \{\frac{b_1}{b_0},\ldots,\frac{b_{i-1}}{b_0}\right\}$.

\medskip 

\noindent If $k=e_{i-1}-n_i>0$ then  $\ord \Gamma^{(i)}(0,y)=\ord \Gamma_2^{(i)}(0,y)=n_1\cdots n_{i-1}$. By  the fifth statement of  Theorem \ref{main} we conclude that $\Gamma^{(i)}$ is irreducible and by the  second statement of  Theorem \ref{main} we get 
$\ch \, \Gamma^{(i)}=\left \{\frac{b_1}{b_0},\ldots,\frac{b_{i}}{b_0}\right\}$.
\end{proof}

\begin{Remark} 
\label{equisingular} The characteristic sequence  or equivalently the set of characteristic exponents determines the equisigularity class (in the sense of Zariski) of an irreducible singular curve. 
Since the contact orders of the irreducible power series~$f$ and the branches of its higher order polars $\frac{\partial^k} {\partial y^k}f(x,y)$, for  $k<e_{h-1}$ are precisely the characteristic exponents of $f$, they determine the equisingularity class of $f(x,y)=0$. 
The case $k=1$ is well-known after \cite[page 110]{Merle}.
 \end{Remark}
 
\medskip

\noindent Remember that if $f(x,y)\in \bC[[x,y]]$ is irreducible with $\ch \, f=\left\{\frac{b_1}{b_0}, \ldots, \frac{b_h}{b_0}\right\}$ then an irreducible power series $g$ is called an $i-1$-{\em semi-root} of $f$ if $\ch \, g=\left\{\frac{b_1}{b_0}, \ldots, \frac{b_{i-1}}{b_0}\right\}$ and  $\cont(f,g)=\frac{b_{i}}{b_0}$. In the language of resolution of singularities, a branch with characteristic contact is a semi-root if and only if its  strict transform is a curvetta of the divisor corresponding to an end vertex of valency1 (different to the root) of the dual resolution graph of $f(x,y)=0$.

\medskip

\noindent Assume that $1\leq i\leq i_k$. Proceeding as in the proof of Corollary \ref{coro1} we can show that 
if $k+1\equiv 0$ (mod $n_i$) then  $\Gamma^{(i)}_1$ is an $(i-1)$-semi-root of $f$.
\medskip

\noindent We call an irreducible power series $g$  an $i$-{\em threshold semi-root} of $f$ if $\ch \, g=\left\{\frac{b_1}{b_0}, \ldots, \frac{b_{i}}{b_0}\right\}$ and  $\cont(f,g)=\frac{b_{i}}{b_0}$. Remark that an  $i$-{\em threshold semi-root} of $f$ is not an $i$-semi-root since its contact with $f$ is  not hight enough.

\medskip

\noindent The irreducible factors of $\Gamma^{(i)}_2$ are  {\em $i$-threshold semi-roots}.

\medskip

\noindent Figure 2 is the schematic picture of the dual resolution graph of the curve $f\cdot \Gamma^{(i)}_2 \cdot f^{(i-1)}\cdot f^{(i)}$, where $f^{(j)}$ is a $j$-semi-root of $f$, $E_j$ denotes the $j$-th rupture point and $\overline g$ means the strict transform of $g=0$. Here we assume that $x=0$ and $f=0$ are transverse.

\begin{center}
\begin{picture}(100,70)(0,0){
\put(-40,0){\line(1,0){138}}}
\put(100,0){\vector(1,1){20}}\put(120,20){$\overline f$}
\put(-42,-2.5){$\bullet$}
\put(-14.5,0){\line(0,1){42}} \put(-17,38){$\bullet$}
\put(-17,-2.5){$\bullet$}\put(-17,-10){$E_1$} 
\put(20.5,0){\line(0,1){42}}
\put(18,-2.5){$\bullet$}\put(19,-10){$E_i$}
\put(20,-1){\vector(1,1){20}}
\put(20,-1){\vector(1,2){20}}
\put(38,22){$\vdots$} 
\put(41.5,25){$\overline{\Gamma^{(i)}_2}$} 
\put(18,38){$\bullet$} 
\put(59.5,0){\line(0,1){42}}
\put(57,-2.5){$\bullet$}\put(56,-10){$E_{i+1}$}    
\put(57,38){$\bullet$}    
\put(59,40){\vector(1,1){20}}\put(79,60){$\overline{f^{(i)}}$} 
\put(20,40){\vector(1,1){20}}\put(40,60){$\overline{f^{(i-1)}}$} 
\put(-4,38){$\dots$} 
\put(99.5,0){\line(0,1){42}}          
\put(74,38){$\dots$}         
\put(97,-2.5){$\bullet$}\put(96,-10){$E_h$}                   
\put(97,38){$\bullet$}       
\put(39,-25.5) {\hbox{\rm Figure 2}} ;              
\end{picture}
\end{center}

\vspace{0.6cm}

\begin{Example}
\label{Ex1}
Consider $f(x,y)=(y^3-x^4)^2-x^9\in \bC[[x,y]]$. 
The curve $f(x,y)=0$ is irreducible of characteristic $(b_0,b_1,b_2)=(6,8,11)$.
Then for the first partial derivative 
$\frac{\partial} {\partial y}f(x,y)=\Gamma^{(1)}\Gamma^{(2)}$
where  $\Gamma^{(1)}=6y^2$, $\Gamma^{(2)}=y^3-x^4$.
If $2\leq k \leq 5$ then $\frac{\partial^k}{\partial y^k}f(x,y)=\Gamma^{(1)}$. 
For $k=2$ we have $\Gamma^{(1)}=\Gamma^{(1)}_1\Gamma^{(1)}_2$, where $\Gamma^{(1)}_1=6y$ and $\Gamma^{(1)}_2=5y^3-2x^4$. 
For $k=3$ the factor $\Gamma^{(1)}_1$ is a unit, while for $k\in \{4,5\}$ the factor $\Gamma^{(1)}_2$ is a unit. 
In Figure 3  the dual resolution graph of the curve $y\cdot (y^3-x^4) \cdot (5y^3-2x^4) \cdot f$ is drawn. Here $E_j$ denotes the $j$th-divisor and $\overline g$ means the strict transform of $g=0$.

\begin{center}
\begin{tikzpicture}[scale=2]
\draw [-](0,0) -- (2,0); 
\node[draw,circle,inner sep=1.5pt,fill=black] at (0,0){};  
\node[draw,circle,inner sep=1.5pt,fill=black] at (1,0) {};
\node[draw,circle,inner sep=1.5pt,fill=black] at (2,0) {};
\draw [-](1,0) -- (1,1); 
\node[draw,circle,inner sep=1.5pt,fill=black] at (1,0.5) {};
\node[draw,circle,inner sep=1.5pt,fill=black] at (1,1) {};
\draw [-](2,0) -- (2,1); 
\node[draw,circle,inner sep=1.5pt,fill=black] at (2,0.5) {};
\node[draw,circle,inner sep=1.5pt,fill=black] at (2,1) {};
\draw (0,-0.1) node[below]{$E_1$} ;
\draw (1,-0.1) node[below]{$E_4$} ;
\draw (2,-0.1) node[below]{$E_7$} ;
\draw (1,0.5) node[left]{$E_3$} ;
\draw (1,1) node[left]{$E_2$} ;
\draw (2,0.5) node[left]{$E_6$} ;
\draw (2,1) node[left]{$E_5$} ;
\draw [->](2,0)--(2.5,-0.5);
\draw [->](1,0) -- (1.5,-0.5); 
\draw [->](1,1) -- (1.5,0.5); 
\draw [->](2,1) -- (2.5,0.5); 
\draw (2.5,-0.5) node[right]{$\overline f$} ;
\draw (1.5,-0.5) node[below]{$\overline{5y^3\!-\!2x^4}$} ;
\draw (1.5,0.5) node[below]{$\overline{y}$} ;
\draw (2.5,0.5) node[right]{$\overline{y^3\!-\!x^4}$} ;
\draw (1.5,-1.2) node[above]{\hbox{\rm Figure 3}} ;

\end{tikzpicture}
\end{center}
\end{Example}

\medskip

\noindent {\bf Acknowledgements:} The authors thank Bernard Teissier 
for suggesting  the name {\em threshold semi-root}.

\medskip

\noindent {\small Evelia Rosa Garc\'{\i}a Barroso\\
Departamento de Matem\'aticas, Estad\'{\i}stica e I.O.\\
Secci\'on de Matem\'aticas, Universidad de La Laguna\\
Apartado de Correos 456\\
38200 La Laguna, Tenerife, Espa\~na\\
e-mail: ergarcia@ull.es}

\medskip

\noindent {\small   Janusz Gwo\'zdziewicz\\
Institute of Mathematics\\
Pedagogical University of Cracow\\
Podchor\c a{\accent95 z}ych 2\\
PL-30-084 Cracow, Poland\\
e-mail: gwozdziewicz@up.krakow.pl}
\end{document}